\documentclass{amsart}

\usepackage{amssymb, latexsym,pdfsync}
\theoremstyle{plain}
\newtheorem{theorem}{Theorem}
\newtheorem{corollary}{Corollary}
\newtheorem{lemma}{Lemma}

\theoremstyle{definition}
\newtheorem{definition}{Definition}
\newtheorem{remark}{Remark}

\newcommand{\SSS}{\mathbb{S}}
\newcommand{\NN}{\mathbb{N}}

\newcommand{\FF}{\mathbb{F}}
\newcommand{\Fq}{\mathbb{F}_q}
\newcommand{\Fqn}{\mathbb{F}_{q^n}}

\newcommand{\Fqd}{\mathbb{F}_{q^d}}

\newcommand{\rank}{\mathrm{rank}}

\begin{document}

\title[Semifields from Skew Polynomial Rings]
{Semifields from\\ 
Skew Polynomial Rings}
\author{Michel Lavrauw
and John Sheekey}
\address{
Michel Lavrauw: Department of Management and Engineering\\
Universit\`a di Padova\\
Italy
\newline John Sheekey: Mathematics Department\\
University College\\
Belfield, Dublin 4\\
Ireland
}
\email{michel.lavrauw@unipd.it, john.sheekey@ucd.ie}
\thanks{The first author acknowledges the support of the Fund for 
Scientific Research - Flanders (FWO). The second author was supported by Claude Shannon Institute, Science Foundation
Ireland Grant 06/MI/006.}
\begin{abstract}
Skew polynomial rings were used to construct finite semifields by Petit in \cite{Petit}, following from a construction of Ore and Jacobson of associative division algebras. Johnson and Jha \cite{JHJO1989} later constructed the so-called {\it cyclic} semifields, obtained using irreducible semilinear transformations. In this work we show that these two constructions in fact lead to isotopic semifields, show how the skew polynomial construction can be used to calculate the nuclei more easily, and provide an upper bound for the number of isotopism classes, improving the bounds obtained by Kantor and Liebler in \cite{KALI2008} and implicitly by Dempwolff in \cite{Dempwolffprep}.

\end{abstract}
\maketitle

\section{INTRODUCTION}
A \emph{semifield} is a division algebra, where multiplication is not necessarily associative. Finite nonassociative semifields of order $q$ are known to exist for each prime power $q=p^n>8$, $p$ prime, with $n>2$. The study of semifields was initiated by Dickson in \cite{Dickson1906} and by now many constructions of semifields are known. We refer to the next section for more details.

In 1933, Ore  \cite{Ore1933} introduced the concept of \emph{skew-polynomial rings} $R = K[t;\sigma]$, where $K$ is a field, $t$ an indeterminate, and $\sigma$ an automorphism of $K$. These rings are associative, non-commutative, and are left- and right-Euclidean. Ore (\cite{Ore1932}, see also Jacobson \cite{Jacobson1934}) noted that multiplication in $R$, modulo right division by an irreducible $f$ contained in the centre of $R$, yields associative algebras without zero divisors. These algebras were called {\it cyclic algebras}. We show that the requirement of obtaining an associative algebra can be dropped, and this construction leads to nonassociative division algebras, i.e. semifields. Subsequent to the writing of this paper, it was brought to the authors' attention that this was noted by Petit \cite{Petit} in 1966 (see also Wene \cite{Wene}).

In 1989, Jha and Johnson \cite{JHJO1989} gave a construction for semifields, 
using irreducible \emph{semilinear transformations}. These semifields were called {\it cyclic semifields}.

In this work we show that the constructions from \cite{JHJO1989} and \cite{Petit} lead to isotopic semifields.
This is Theorem \ref{thm:correspondence2} and Theorem \ref{thm:correspondence3} and can be formulated as follows.

\begin{theorem}
Each cyclic semifield is isotopic to a semifield constructed as a quotient in a skew polynomial ring, and conversely, each semifield constructed as a quotient in a skew polynomial ring is isotopic to a cyclic semifield.
\end{theorem}

We also investigate the number of isotopism classes of semifields of order $q^{nd}$, obtained from an irreducible $f$ of degree $d$ in the skew polynomial ring $R=\Fqn[t;\sigma]$, where $Fix(\sigma) = \Fq$. We denote this number by $A(q,n,d)$.

In \cite{KALI2008} Kantor and Liebler provided an upper bound for the number of isotopism classes of semifields arising from semilinear transformations. This bound has recently been improved (implicitly) by Dempwolff in \cite{Dempwolffprep}. We further improve on this bound by proving an upper bound for $A(q,n,d)$.

We conclude the introduction with the statement of this bound.
Let
\[
I(q,d) := \{ f \in \Fq[y] ~|~ \textrm{$f$ monic, irreducible, degree $d$}\},
\]
and let $G$ be the semidirect product of $\Fq^\times$ and $Aut(\Fq)$, and define the action of $G$ on $I(q,d)$ in the following way
\[
f(y)^{(\lambda,\rho)}:= \lambda^{-d}f^{\rho}(\lambda y)
\]
where $\lambda \in \Fq^{\times}$, $\rho \in Aut(\Fq)$. If $q = p^h$ for $p$ prime, $G$ has order $h(q-1)$.
We will prove the following theorem.

\begin{theorem} The number of isotopism classes of semifields of order $q^{nd}$ obtained from $\Fqn[t;\sigma]$ is less or equal to the number of $G$-orbits on $I(q,d)$.
\end{theorem}
We denote this number of orbits by $M(q,d)$. 
This number lies in the interval
\[
\frac{q^d - \theta}{hd(q-1)}\leq M(q,d) \leq \frac{q^d - \theta}{d},
\]
where $\theta$ denotes the number of elements of $\Fqd$ contained in a subfield $\FF_{q^e}$ for $e|d$, and $q = p^h$, where $p$ is prime.

\section{Finite semifields}
In this section we collect the terminology of the theory of finite semifields, used in the remainder of the paper. For more details on the subject we refer to \cite{Knuth1965}, \cite{Kantor2006} and \cite{LaPo2011}.
A {\it finite semifield}\index{semifield}\index{finite semifield} $\mathbb S$ is a finite algebra with at least two elements,
and two binary operations $+$ and $\circ$, satisfying the following axioms. 
\begin{itemize}
\item[(S1)] $({\mathbb{S}},+)$ is a group with neutral $0$.
\item[(S2)] $x\circ(y+z) =x\circ y + x\circ z$ and $(x+y)\circ z = x\circ z + y
\circ z$, for all $x,y,z \in {\mathbb{S}}$.
\item[(S3)] $x\circ y =0$ implies $x=0$ or $y=0$.
\item[(S4)] $\exists 1 \in {\mathbb{S}}$ such that $1\circ x = x \circ 1 = x$,
for all $x \in {\mathbb{S}}$.
\end{itemize}

One easily shows that the additive group of a finite semifield is elementary abelian,
and the exponent of the additive group of $\mathbb S$ is called the {\it
characteristic} of $\mathbb S$. Contained in a finite semifield are the following
important substructures, all of which are isomorphic to a finite field. The {\it
left nucleus}\index{left nucleus} ${\mathbb{N}}_l({\mathbb{S}})$, 
{\it the middle nucleus}\index{middle nucleus}
${\mathbb{N}}_m({\mathbb{S}})$, and the {\it right nucleus}\index{right nucleus}
${\mathbb{N}}_r({\mathbb{S}})$ are defined as follows: 
\begin{equation}
{\mathbb{N}}_l({\mathbb{S}}):=\{x~:~ x \in {\mathbb{S}} ~|~ x \circ (y\circ
z)=(x\circ y)\circ z, ~\forall y,z \in {\mathbb{S}}\}, 
\end{equation}
\begin{equation}
{\mathbb{N}}_m({\mathbb{S}}):=\{y~:~ y \in {\mathbb{S}} ~|~ x \circ (y\circ
z)=(x\circ y)\circ z, ~\forall x,z \in {\mathbb{S}}\}, 
\end{equation}
\begin{equation}
{\mathbb{N}}_r({\mathbb{S}}):=\{z~:~ z \in {\mathbb{S}} ~|~ x \circ
(y\circ z)=(x\circ y)\circ z, ~\forall x,y \in {\mathbb{S}}\}.
\end{equation}
The intersection $\NN(\SSS)$ of the nuclei is called the \emph{associative centre}\index{associative center}, and the elements of $\NN(\SSS)$ which commute with all other elements of $\SSS$ form the \emph{centre}\index{center} $Z(\SSS)$. If there is no confusion, we denote these subfields by $\NN_l$, $\NN_m$, $\NN_r$, $Z$.

Two semifields $\mathbb S$ and $\hat{\mathbb{S}}$ are called {\it isotopic} \index{isotopic} if
there exists a triple  $(F,G,H)$ of non-singular linear transformations from
$\mathbb S$ to $\hat{\mathbb{S}}$  such that $x^F\hat\circ y^G = (x\circ y)^H$,
for all $x,y,z \in {\mathbb{S}}$. The triple $(F,G,H)$ is called an {\it
isotopism}\index{isotopism}.

\section{Semifields from skew-polynomial rings}

In this section we use an \emph{irreducible} polynomial in a skew polynomial ring to construct a semifield.
We start with some definitions and properties of skew polynomial rings. For a more detailed description we refer to Ore \cite{Ore1933}.

\begin{definition} Let $K$ be a field, and $\sigma$ an automorphism of $K$. Define the \emph{skew polynomial ring} $R = K[t;\sigma]$ to be the set of polynomials in $t$ with coefficients in $K$, where addition is defined termwise, and multiplication is defined by $ta = a^{\sigma}t$ for all $a \in K$.
\end{definition}

We say that an element $f$ is \emph{irreducible} in $R$ if there do not exist any $a,b \in R$ with $\deg(a),\deg(b) < \deg(f)$ such that $f = ab$.

\begin{theorem}[Ore \cite{Ore1933}]
\label{Oreprop} Let $R$ be a skew-polynomial ring. Then
\begin{enumerate}
	\item
		multiplication in $R$ is associative and $R$ satisfies both distributive laws;
	\item
		multiplication in $R$ is not commutative unless $\sigma$ is the identity automorphism;
	\item
		$R$ is left- and right-Euclidean;
	\item
		$R$ is a left- and right-principal ideal domain;
	\item
		the centre of $R= K[t;\sigma]$ is $F[t^n;\sigma] \simeq F[y]$, where $F$ is the fixed field of $\sigma$ and the isomorphism maps $t^n$ to $y$;
	\item
		if $f_1,f_2, \ldots ,f_r,g_1,g_2,\ldots,g_s$ are irreducible elements of $R$, and 
		\[
		f_1 f_2 \ldots f_r = g_1 g_2 \ldots g_s
		\]
		then $r=s$ and there is a permutation $\pi \in S_r$ such that $\deg(f_i) = \deg(g_{\pi(i)})$ for all $i$.
\end{enumerate}
\end{theorem}

For this paper we will set $K = \Fqn$, and let $F = \Fq$ be the fixed field of $\sigma$. The properties of skew polynomial rings allow us to define a semifield in the following way. The result is not new, see Remark \ref{discuss} below. We include a proof for the sake of completeness.

\begin{theorem} Let $V$ be the vector space consisting of elements of $R$ of degree strictly less than $d$. Let $f \in R$ be irreducible of degree $d$. Define a multiplication $\circ_f$ on $V$ by
\[
a \circ_f b := ab \mod_r f
\]
where juxtaposition denotes multiplication in $R$, and '$\mod_r$' denotes remainder on right division by $f$.
Then $\SSS_f = (V,\circ_f)$ is semifield of order $q^{nd}$.
\end{theorem}
\begin{proof}
This multiplication is well defined, as $R$ is right-Euclidean. We check that $\SSS_f$ has no zero divisors. Suppose
$a,b\in \SSS_f$, and $a \circ_f b=0$. This implies $\exists h\in \SSS_f$ such that $ab=hf$. Comparing degrees, part (6) of the previous theorem gives a contradiction unless $a$ or $b$ is the zero polynomial. The other properties of a semifield are easily verified. Obviously $\SSS_f$ has order $q^{nd}$.
\end{proof}

\begin{remark}
Note that for any $0 \ne \alpha \in K$, the polynomials $f$ and $\alpha f$ define the same semifield.

Note that defining the multiplication using remainder on \emph{left} division by $f$ also defines a semifield. However, in Corollary \ref{thm:leftdivision} we will show that the semifields obtained are anti-isomorphic.
\end{remark}

\begin{remark}
For the rest of this paper we will write {\it mod} for {\it mod$_r$} unless otherwise stated, and write \emph{divides} for \emph{right divides}.
\end{remark}

In  \cite{Ore1932} Ore introduced the following notion of {\it eigenring} (called the {\it normalizer} by Jacobson in \cite{Jacobson1934}).

\begin{definition} Let $f$ be a monic irreducible element of $R$ of degree $d$. Define the \emph{eigenring} of $f$ by
\[
E(f) = \{ u \in R ~|~ \deg(u) < d,~f~{\textrm{divides}}~fu \}
\]
\end{definition}

\begin{remark}
\label{discuss}
Ore  and Jacobson, when studying \emph{cyclic algebras}, each considered structures obtained from the vector space of residue classes of $R = K[t;\sigma]$ modulo a left ideal $Rf$. As they were interested only in associative algebras, they restricted their attention to the eigenring $E(f)$. They each proved (in different ways) the following theorem (\cite{Ore1932}, p. 242 and \cite{Jacobson1934}, p. 201-202):
\medskip
\begin{center}
{\it
If $f$ is irreducible in $R$, then $E(f)$ is a[n associative] division algebra.
}
\end{center}

As we have seen above, if we choose a specific representative of each residue class (the unique element of degree less than $\deg(f)$), then the structure $\SSS_f$ obtained is a non-associative algebra. The theorem then trivially extends to:
\begin{center}
{\it
If $f$ is irreducible in $R$, then $\SSS_f$ is a division algebra.
}
\end{center}
The proof relies only on the theorem of Ore (Theorem \ref{Oreprop} above). Hence it is perhaps fair to say that the construction of the semifields $\SSS_f$
was, in essence, known to Ore and Jacobson.

This construction was then explicitly formulated by Petit \cite{Petit} in 1966. As this construction is perhaps not well known, and as some of the tools used are required for later results, we include proofs of some of the results contained therein.
\end{remark}
\section{Nuclei}
\label{sect:nuclei}
We now investigate the nuclei of the above defined semifields. These results can be found in \cite{Petit}. 
 \begin{theorem}(\cite{Petit},2)
\label{thm:rnucleus}
Let $f$ be a \emph{monic} irreducible element of $R$ of degree $d$, and let $\SSS_f$ be the semifield as defined above. Then
\[
N_r(\SSS_f) = E(f)
\]
and
\[
E(f) = \SSS_f \Leftrightarrow f \in Z(R)
\]
where $Z(R)$ denotes the centre of $R$.
\end{theorem}

\begin{proof}
First we will prove the second assertion. Suppose $E(f) = \SSS_f$. Let
\[
f = \sum_{i=0}^d f_i t^i
\]
where $f_i \in K$, and $f_d = 1$ as $f$ is monic. As $t\in E(f)$ by assumption, we must have $ft \equiv 0 \mod f$. But then
\begin{align*}
ft \mod f &= ft - tf\\
					&= \sum_{i=0}^d (f_i - f_i^{\sigma})t^i\\
					&= 0,
\end{align*}
implying that $f_i = f_i^{\sigma}$ for all $i$, and so $f_i \in F$ for all $i$. Now as $\alpha \in E(f)$ for all $\alpha \in K$, we have
\begin{align*}
f \alpha \mod f &= f\alpha - \alpha^{\sigma^d}f\\
					&= \sum_{i=0}^d (\alpha^{\sigma^i} - \alpha^{\sigma^d})f_i t^i\\
					&= 0,
\end{align*}
implying that for each $i$ we have $f_i=0$ or $\alpha^{\sigma^i}= \alpha^{\sigma^d}$ for all $a\in K$. As $f$ is irreducible, we must have $f_0 \ne 0$ (for otherwise $t$ would divide $f$). Hence if $f_i \ne 0$, we have $\alpha^{\sigma^i}= \alpha$ for all $\alpha \in K$, and so $\sigma^i = \mathrm{id}$. Hence if $f_i \ne 0$ then $n$ divides $i$. Therefore $f \in F[t^n;\sigma] = Z(R)$, as claimed.

Conversely, if $f \in Z(R)$ then clearly $fu = uf$  is divisible by $f$ for all $u$, and so $E(f) = \SSS_f$.

We now show that $\NN_r(\SSS_f) = E(f)$. For any $a,b,c \in R$ of degree less than $d = \deg(f)$ we can find unique $u,v,w,z \in R$ of degree less than $d$ such that
\[
ab = uf+v,~\mbox{and}
\]
\[
bc = wf +z,
\]
i.e. $a\circ_f b = v$, $b \circ_f c = z$. Then
\[
(a \circ_f b) \circ_f c = v \circ_f c = vc \mod f,
\]
while
\[
a \circ_f (b \circ_f c) = a \circ_f z = az \mod f.
\]
But as $R$ is associative, we have that
\[
ufc+vc = (ab)c = a(bc) = awf+az,
\]
and hence 
\[
az = ufc + vc \mod f.
\]
Therefore
\[
(a \circ_f b) \circ_f c = a \circ_f (b \circ_f c) \Leftrightarrow ufc = 0 \mod f.
\]
Let $c$ be in the right nucleus. One can choose $a,b$ such that $u= 1$. Then $fc = 0 \mod f$, implying that $c \in E(f)$. Conversely, if $c \in E(f)$ then $ufc = 0 \mod f$ for all $u$, and hence $c$ is in the right nucleus, as claimed.
\end{proof}
Hence we get the following corollary:
\medskip
\begin{corollary}(Petit, \cite{Petit}, 9)
$\SSS_f$ is associative if and only if $f \in Z(R)$.
\end{corollary}
We will see in Lemma \ref{thm:fullmatrix} that if $K$ is a finite field, and $\sigma$ is not the identity automorphism, then every element of $Z(R)$ is reducible. This is also implied by the Wedderburn-Dickson theorem, for otherwise we would obtain a non-commutative finite division algebra. Note however that such elements can exist over infinite fields.

\begin{theorem}(\cite{Petit},2)
\label{thm:lmnuclei}
Suppose $f$ is a monic irreducible element of $R= K[t;\sigma]$ such that $f \notin Z(R)$. The left and middle nuclei of $\SSS_f$ are given by
\[
\NN_l(\SSS_f) = \NN_m(\SSS_f) = (K).1
\]
i.e. they are the set of constant polynomials, and the centre is 
\[
Z(\SSS_f) = (F).1.
\]
\end{theorem}

\begin{proof}
Let $a,b,c \in R$ be of degree less than $d$, and $u,v,w,z$ be as defined in the proof of Theorem \ref{thm:rnucleus}. We saw that $(a \circ_f b) \circ_f c = a \circ_f (b \circ_f c) \Leftrightarrow ufc = 0 \mod f$.

We show that an element is in the left nucleus if and only if it has degree zero. First suppose $a$ has degree zero. Then for any $b$, $ab$ has degree strictly less than $d$, and hence $u=0$ for all $b$. Therefore $ufc = 0 \mod f$ for all $b,c$, and so $a \in \NN_l$.

Suppose now $\deg(a)=r>0$, and let $a_r$ be the leading coefficient of $a$. Let $b = \frac{1}{a_r ^{-\sigma^r}} t^{d-r}$. Then $ab$ is monic, and has degree $d$, and so $u = 1$. Let $c$ be some element not in $E(f)$, i.e. $fc \ne 0 \mod f$. We know that such an element exists as $f \notin Z(R)$. Then $ufc = fc \ne 0 \mod f$, and so $a \notin \NN_l$.

The proof for $\NN_m$ is similar.

The centre is a subfield of $\NN_l$, and so consists of all constant polynomials which commute with $t$. Since $ta = a^{\sigma}t$ for all $a\in K$, the centre is therefore equal to the fixed field of $\sigma$, which is $F$.
\end{proof}

Later we will show that $|N_r(\SSS_f)| = q^d$. The nuclei of $\SSS_f$ were calculated in a different way by Dempwolff in \cite{Dempwolff2011}, when he calculated the nuclei of cyclic semifields, which we will show in Section \ref{sect:cyclicequiv} to be equivalent to this construction.

Hence if two semifields defined by polynomials $f \in K[t,\sigma]$ and $f' \in K'[t,\sigma']$ are isotopic, then $K = K'$, $\deg(f) = \deg(f')$, and $\sigma$ and $\sigma'$ have the same fixed field (i.e. the same order). In the next section we will investigate when two such semifields are isotopic. 

\section{Isotopisms between semifields $\SSS_f$}
In this section we will first consider some properties of skew polynomial rings, which we will use later to obtain isotopisms of the above defined semifields.

\begin{lemma}
Let $\phi$ be an automorphism of $R = K[t;\sigma]$, where $\sigma$ is not the identity automorphism. Then
\[
\phi(f) = f^{\rho}(\alpha t)
\]
where $\rho \in Aut(K)$ and $\alpha \in \Fqn^{\times}$.
\end{lemma}

\begin{proof}
As $\phi$ is bijective, it preserves the degree of elements of $R$. Let $\rho$ be the field automorphism obtained by the restriction of $\phi$ to $K$, and assume $\phi(t)= \alpha t+ \beta$, $\alpha,\beta \in K$, $\alpha \ne 0$. Choose $\gamma \in K$ such that $\gamma^{\sigma}\ne \gamma$. Computing $\phi(t)\phi(\gamma)= \phi(t\gamma)=\phi(\gamma^{\sigma}t) = \phi(\gamma^{\sigma})\phi(t)$, we see that $\beta = 0$, and the assertion follows.
\end{proof}

Automorphisms of $R$ can be used to define isomorphisms between semifields.

\begin{theorem}\label{thm:isom}
Let $f$ be an irreducible of degree $d$ in $R$. Let $\phi$ be an automorphism of $R$. Define $g = \phi(f)$. Then $\SSS_f$ and $\SSS_g$ are isomorphic, and 
\[
\phi(a \circ_f b) = \phi(a) \circ_g \phi(b).
\] 
\end{theorem}

The proof is left to the reader. We now consider another type of isotopism between these semifields.

\begin{definition}
Let $f$ and $g$ be monic irreducibles of degree $d$ in $R$. We say that $f$ and $g$ are \emph{similar} if there exists a non-zero element $u$ of $R$ of degree less than $d$ such that
\[
gu \equiv 0 \mod f.
\]
\end{definition}

\begin{theorem}\label{thm:similar}
Suppose $f$ and $g$ are similar. Then $\SSS_f$ and $\SSS_g$ are isotopic, and
\[
(a \circ_g b)^{H}= a \circ_f  b^{H}
\]
where $b^H = b\circ_f u$, $gu \equiv 0 \mod f$.
\end{theorem}

\begin{proof}
Let $a \circ_g b = ab - vg$. Then 
\[
(a \circ_g b)^H = (ab-vg)^H = (ab-vg)u \mod f = (abu - vgu) \mod f  \equiv abu \mod f
\]
as $gu \equiv 0 \mod f$.

Let $b \circ_f u = bu - wf$. Then
\[
a \circ_f b^H = a \circ_f (b \circ_f u) = a \circ_f (bu-wf)
\]
\[
= a(bu-wf) \mod f \equiv abu \mod f
\]
and the result holds.
\end{proof}

In  \cite{Jacobson1937} Jacobson investigated when two skew polynomials are similar. We include a proof here for completeness, and because some of the concepts introduced will be of use later in this paper.

\begin{definition} Let $f \in R$ be irreducible of degree $d$. Define the \emph{minimal central left multiple} of $f$, denoted by $mzlm(f)$, as the monic polynomial of minimal degree in the centre $Z \simeq \Fq[t^n;\sigma] \simeq \Fq[y]$ that is right-divisible by $f$.
\end{definition}

In \cite{Giesbrecht1998} Giesbrecht showed that $mzlm(f)$ exists, is unique, has degree $d$ and is irreducible when viewed as an element of $\Fq[y]$ (which we state in the next lemma). Note that this is related to the \emph{bound} of $f$: if $t$ does not divide $f$, then $R.mzlm(f)$ is the largest two-sided ideal of $R$ contained in the left ideal $R.f$. See for example \cite{Jacobson1943}.

\begin{lemma}[\cite{Giesbrecht1998}]
Let $f \in R$ be irreducible of degree $d$. Let $mzlm(f) =\hat{f}(t^n)$ for some $\hat{f} \in \Fq[y]$. Then $\hat{f}$ is irreducible.
\end{lemma}

\begin{lemma}\label{thm:fullmatrix}
Let $h$ be an element of $R$ such that $h = \hat{h}(t^n)$, where $\hat{h} \in \Fq[y]$ is monic, irreducible and has degree $d$ in $y$ and $\hat{h} \ne y$. Then
\begin{enumerate}
	\item[(1)]
		\[
		\frac{R}{Rh} \simeq M_n(\Fqd);
		\]
	\item[(2)]
		any irreducible divisor $f$ of $h = \hat{h}(t^n)$ has degree $d$;
	\item[(3)]
		if $A$ denotes the isomorphism of part (1), and $f$ is an irreducible (right) divisor of $h$, then the matrix $A(f+Rh)$ has rank $n-1$.
\end{enumerate}
\end{lemma}
By abuse of notation we will write $A(a) = A(a+Rh)$ for $a \in R$.
\begin{proof}
(1) First we show that $Rh$ is a maximal two-sided ideal in $R$. For suppose there exists some $g \in R$ such that $Rg$ is a two-sided ideal, $\deg(g) < \deg(h)$ and $Rh \subset Rg$. Then
\[
g = \hat{g}(t^n) t^s
\]
for some $\hat{g}\in \Fq[y]$ (see for example \cite{JacobsonBook} Theorem 1.2.22). As $t$ does not divide $h$, we must have that $s = 0$, and
\[
h = ag
\]
for some $a \in R$. As $h$ and $g$ are in the centre of $R$, $a$ must also be in the centre of $R$, and so $a = \hat{a}(t^n)$ for some $\hat{a} \in \Fq[y]$. But then
\[
\hat{h}(y) = \hat{a}(y) \hat{g}(y)
\]
As $\hat{h}$ is irreducible in $\Fq[y]$, we must have $\hat{g} \in \Fq$, and so $g \in \Fq$. Therefore $Rg = R$, proving that $Rh$ is maximal.

It follows that $\frac{R}{Rh}$ is a finite simple algebra and hence isomorphic to a full matrix algebra over its centre (\cite{Lang} Chapter 17). It is easily shown (see for example \cite{Giesbrecht1998}, proof of Theorem 4.3) that the centre $Z\left (\frac{R}{Rh}\right )$ is the image of the centre of $R$, and is given by
\[
Z\left (\frac{R}{Rh} \right ) = \frac{Z(R)+Rh}{Rh} \simeq \frac{\Fq[y]}{\Fq[y] \hat{h}(y)} \simeq {\Fqd}
\]
as $\hat{h}$ is a degree $d$ irreducible in $\Fq[y]$.

As the dimension of $\frac{R}{Rh}$ as a vector space over $\Fq$ is $n^2 d$, we see that 
\[
\frac{R}{Rh} \simeq M_n(\Fqd)
\]
as claimed.

(2) Let $f$ be an irreducible divisor of $h$, and let $r = \deg(f)$. Then $f$ generates a maximal left ideal in $R$, and also in $\frac{R}{Rh}$. This maximal left ideal $\left(\frac{R}{Rh}\right) f$ is then $(n^2 d - nr)$-dimensional over $\Fq$.

By part (1), we know that $\frac{R}{Rh}$ is isomorphic to $M := M_n(\Fqd)$. It is well known that maximal left ideals in $M$ are all of the form $Ann_M(U)$ for some 1-dimensional space $U < (\Fqd)^n$, and are $(n^2-n)$-dimensional over $\Fqd$, and hence $(n^2-n)d$-dimensional over $\Fq$. Therefore $r = d$, as claimed.

(3) The left ideal $M.A(f)$ is equal to $Ann_M(Ker(A(f)))$, and so $A(f)$ has rank $n-1$ as claimed.
\end{proof}

\begin{remark} Hence the number of monic irreducible elements of degree $d$ in $R$ can be seen to be
\[
N(q,d) \left( \frac{q^{nd}-1}{q^d - 1} \right).
\]
This was calculated by Odoni \cite{Odoni1999}, and is an upper bound for $A(q,n,d)$. However, we will see that this is far from optimal.
\end{remark}

\begin{lemma}\label{lem:size_eigenring}
If $f \in R$ is irreducible of degree $d$, then $|E(f)| = q^d$.
\end{lemma}

\begin{proof}
Let $u$ have degree less than $nd$, and let $u = af + u'$ for $\deg(u') < \deg(f)$. Then $fu \equiv 0 \mod f$ if and only if $u' \in E(f)$.
Let $E'$ be the set of all $u+Rh \in R/Rh$ such that $(f+Rh)(u+Rh)=(v+Rh)(f+Rh)$ for some $v+Rh \in R/Rh$. Then $u+Rh \in E'$ if and only if there exists some $v\in R$ such that $fu+Rh = vf+Rh$, which occurs if and only if there exists $v \in R$ such that $fu \equiv vf \mod h$. But then as $f$ divides $h$, we have $fu \equiv vf \mod f \equiv 0 \mod f$. Hence we have that
\begin{align*}
E' 	&= \{(af+u')+Rh : a \in R, \deg(a) <d(n-1),u' \in E(f)\} \\
		&= \frac{(E(f) + Rf)+Rh}{Rh}.
\end{align*}
Hence we have that $|E'| = q^{dn(n-1)}|E(f)|$.

By part (1) of Lemma \ref{thm:fullmatrix}, $\frac{R}{Rh} \simeq M_n(\Fqd) = M$. If $A$ denotes this isomorphism, then by part (3) of Lemma \ref{thm:fullmatrix}, $A(f) := A(f+Rh)$ has rank $n-1$. Let $Ker(A(f)) = <v>$ for $0 \ne v \in (\Fqd)^n$. Then 
\[
u+Rh \in E' \Leftrightarrow A(f)A(u) \in M.A(f) \Leftrightarrow A(u)v = \lambda v
\]
for some $\lambda \in \Fqd$. Then $A(u) - \lambda I \in Ann_M(v)$, and so 
\[
|E'| = q^d |Ann_M(v)| = q^{d(n^2-n+1)}.
\]
Hence from the two expressions for $|E'|$ we get $|E(f)| = q^d$, as claimed.
\end{proof}

\begin{remark}
We see that
\[
E(f) = \{ z \mod f ~:~z \in Z(R) \}
\]
i.e. the remainders of all central elements of $R$ on right division by $f$.
\end{remark}

The sizes of the nuclei and the centre of the semifield $\SSS_f$ now easily follow from
Theorems \ref{thm:rnucleus},\ref{thm:lmnuclei} and Lemma \ref{lem:size_eigenring}.
\begin{theorem} If $f\in R$ is irreducible of degree $d$, then for the nuclei and the centre of semifield $\SSS_f$ we have
$$(\#Z,\#\NN_l,\# \NN_m,\#\NN_r)  = (q,q^n,q^n,q^d).$$
\end{theorem}

The following theorem tells us exactly when two irreducibles are similar.
\begin{theorem}\label{thm:equivalence1}
Let $f$ and $g$ are irreducible in $R$. Then $mzlm(g) = mzlm(f)$ if and only if $f$ and $g$ are similar.
\end{theorem}

\begin{proof} Suppose first that $mzlm(g) = mzlm(f)$. Let $h$ denote $mzlm(f)$, and write $h = af$. Then $\frac{R}{Rh} \simeq M_n(\Fqd)$. As above, let $A$ denote this isomorphism. By Lemma \ref{thm:fullmatrix}, 
\[
\rank(A(f))=\rank(A(g)) = n-1,
\]
and the equality of ranks shows there exist there exist invertible matrices $A(u),A(v)$ such that $A(u)A(f) = A(g)A(v)$. Then $uf \equiv gv \mod h$, so there exists some $b$ such that
\[
gv = uf + bh = uf + baf = (u+ba)f.
\]
We can write $v = v' + cf$, where $\deg(v') < d$ and $v' \ne 0$ (for otherwise, $v = cf$, and so $v$ has non-trivial common divisor with $h$, so $A(v)$ is not invertible). Then
\[
g(v'+cf) = (u+ba)f
\]
\[
\Rightarrow gv' = (u+ba-gc)f \Rightarrow gv' = u'f
\]
and $g$ and $f$ are similar, as claimed.

Suppose now that $f$ and $g$ are similar. By definition, $gu = vf$ for some $u,v$ of degree less than $d$. It can be shown that
\[
mzlm(ab) = mzlm(a)mzlm(b)
\]
if $gcrd(a,b) = 1$. See for example \cite{Giesbrecht1998}. Hence 
\[
mzlm(v)mzlm(f) = mzlm(g)mzlm(v),
\]
and as $mzlm(f)$ and $mzlm(g)$ are irreducible in $\Fq[y]$, by uniqueness of factorization in $\Fq[y]$ the result follows.
\end{proof}

Hence the number of isotopy classes is upper bounded by the number of irreducible polynomials of degree $d$ in $\Fq[y]$. This was proved in a different way by Dempwolff \cite{Dempwolffprep}. The next theorem allows us to further improve this bound.

\begin{definition} Consider the group
\[
G = {\mathrm{\Gamma L}} (1,q) = \{ (\lambda,\rho) ~|~ \lambda \in \Fq^{\times}, \rho \in Aut(\Fq) \}.
\]
Define an action of $G$ on $I(q,d)$ by
\[
f^{(\lambda, \rho)}(y) = \lambda^{-d} f^{\rho}(\lambda y).
\]
\end{definition}

\begin{theorem}\label{thm:isotopism}
Let $f,g \in R$ be irreducibles of degree $d$, with $mzlm(f) = \hat{f}(t^n)$, $mzlm(g) = \hat{g}(t^n)$ for $\hat{f},\hat{g} \in \Fq[y]$. If
\[
\hat{g} = \hat{f}^{(\lambda, \rho)}
\]
for some $\lambda \in \Fq^{\times}$, $\rho \in Aut(\Fq)$, then $\SSS_f$ and $\SSS_g$ are isotopic.
\end{theorem}

\begin{proof}
Choose some $\alpha \in \Fqn$ such that
\[
N_{\Fqn / \Fq}(\alpha) = \lambda.
\]
Then 
\[
(\alpha t)^n = \alpha \alpha^{\sigma} \ldots \alpha^{\sigma^{n-1}} t^n = \lambda t^n.
\]
Define
\[
h(t) = f^{\rho} (\alpha t).
\]
By Theorem \ref{thm:isom}, $\SSS_f$ and $\SSS_h$ are isomorphic. Let $mzlm(h) = \hat{h}(t^n)$. 

Let $\phi$ be the automorphism of $R$ defined by $\phi(a) = a^{\rho}(\alpha t)$. Then as $\phi(f)=h$ and $\hat{f}(t^n) = uf$ for some $u \in R$, 
\[
\phi(\hat{f}(t^n)) = \phi(u)\phi(f) = \phi(u) h.
\]
But
\[
\phi(\hat{f}(t^n)) = \hat{f}^{\rho}((\alpha t)^n) = \hat{f}^{\rho}(\lambda t^n).
\]
As this is in the centre of $R$, and is divisible by $h$, we must have that $\hat{h}(y)$ divides $\hat{f}^{\rho}(\lambda t^n)$, and so, as their degrees are equal and both are monic,
\[
\hat{h}(y) = \lambda^{-d} \hat{f}^{\rho}(\lambda y) = \hat{g}(y).
\]
By Theorem \ref{thm:similar}, as $h$ and $g$ have the same minimal central left multiple, $\SSS_g$ and $\SSS_h$ are isotopic, and hence $\SSS_f$ and $\SSS_g$ are isotopic, as claimed.
\end{proof}

Hence the number of isotopy classes is upper bounded as follows.

\begin{theorem}\label{thm:bound}
The number of isotopism classes of semifields $\SSS_f$ of order $q^{nd}$ obtained from $\Fqn[t;\sigma]$ is less or equal to the number of $G$-orbits on the set of monic irreducible polynomials of degree $d$ in $\Fq[y]$.
\end{theorem}
\begin{proof}
Suppose $f$ and $g$ are two monic irreducible polynomials in $\Fqn[t;\sigma]$ of degree $d$, 
with $mzlm(f) = \hat{f}(t^n)$, $mzlm(g) = \hat{g}(t^n)$ for $\hat{f},\hat{g} \in \Fq[y]$. Then by \cite{Giesbrecht1998},
 $\hat{f}$ and $\hat{g}$ are monic irreducible of degree $d$ in $\Fq[y]$. Moreover, if
$\hat{f}^G=\hat{g}^G$, then by Theorem \ref{thm:isotopism}, $\SSS_f$ and $\SSS_g$ are isotopic.
\end{proof}

In the next section we will relate this construction to the construction of Johnson-Jha, and in the last section we will compare this new bound to existing bounds. 

\section{Cyclic semifields and Endomorphisms of left multiplication}
\label{sect:cyclicequiv}
\begin{definition}
A \emph{semilinear transformation} on a vector space $V = K^d$ is an additive map $T:V \rightarrow V$ such that
\[
T(\alpha v) = \alpha^{\sigma}T(v)
\]
for all $\alpha \in K$, $v \in V$, for some $\sigma \in \mathrm{Aut}(K)$. The set of invertible semilinear transformations on $V$ form a group called the \emph{general semilinear group}, denoted by ${\mathrm{\Gamma L}} (d,K)$.
\end{definition}
Note that choosing a basis for $V$ gives us
\[
T(v) = A(v^{\sigma})
\]
where $A$ is some invertible $K$-linear transformation from $V$ to itself, $\sigma$ is an automorphism of $K$, and $v^{\sigma}$ is the vector obtained from $v$ by applying the automorphism $\sigma$ to each coordinate of $v$ with respect to this basis.

\begin{definition}
An element $T$ of ${\mathrm{\Gamma L}}(d,K)$ is said to be \emph{irreducible} if the only $T$-invariant subspaces of $V$ are $V$ and $\{0\}$.
\end{definition}

\begin{theorem}\label{thm:correspondence}
Let $\SSS_f$ be a semifield defined by an irreducible $f = t^d - \sum_{i = 0}^{d-1} f_i t^i$ in $R$, and let $L_t$ denote left multiplication by $t$ in $\SSS_f$. Then the following properties hold.
\begin{enumerate}
	\item
		$L_t$ is an element of ${\mathrm{\Gamma L}} (d,K)$ with accompanying automorphism $\sigma$.
	\item
		If we write elements $v = \sum_{i = 0}^{d-1} v_i t^i$ of $\SSS_f$ as column vectors $(v_0,v_1, \ldots ,v_{d-1})^t$, then
		\[
L_t (v)= A_f(v^{\sigma}) 
\]
where
\[
A_f = \left(\begin{array}{ccccc}
0 & 0 & \ldots &0& f_0 \\
1 & 0 & \ldots &0& f_1 \\
0 & 1 & \ldots &0& f_2 \\
\vdots & \vdots & \ldots &\vdots& \vdots \\
0& 0 & \ldots &1& f_{d-1} \end{array}\right).
\]
	\item
		If $a = \sum_{i = 0}^{d-1} a_i t^i$, then
		\[
		L_a = a(L_t) = \sum_{i = 0}^{d-1} a_i L_t^i.
		\]
	\item
		The semilinear transformation $L_t$ is irreducible.
\end{enumerate}
\end{theorem}

\begin{proof}
(1) Clearly $L_t$ is linear, as multiplication is distributive. Let $v$ be any vector. If $tv = uf+w$ for some unique $u,w$, $\deg(w) < d$, then $L_t(v) = w$.
Let $\alpha$ be any non-zero element of $\Fqn$. Then
\[
L_t(\alpha v) = t(\alpha v) \mod f = (\alpha ^{\sigma}tv) \mod f = \alpha^{\sigma} (uf+w) \mod f
\]
\[
= (\alpha^{\sigma}uf+\alpha^{\sigma}w) \mod f = \alpha^{\sigma}w = \alpha^{\sigma}L_t(v).
\]

(2) The action of $L_t$ is as follows:
\[
L_t : 1 \mapsto t \mapsto t^2 \mapsto \ldots \mapsto t^{d-1} \mapsto (t^d \mod f) = \sum_{i = 0}^{d-1} f_i t^i
\]
and so $L_t(v) = A_f(v^{\sigma})$ as claimed.

(3) By definition, 
\[
L_a(b) = a \circ_f b = \left (\sum_{i = 0}^{d-1} a_i t^i \right )b \mod f = \sum_{i = 0}^{d-1} a_i (t^i b \mod f)
= \sum_{i = 0}^{d-1} a_i L_{t^i}(b)
\]
while
\[
a(L_t)(b) := \sum_{i=0}^{d-1} a_i L_t^i(b).
\]
Hence it suffices to show that $L_t^i(b) = L_{t^i}(b)$ for all $i$.
Suppose $L_t^i(b) = (t^i b) \mod f$ for some $i$. Let $L_t^i(b) = b'$. Then $t^i b = cf + b'$ for some $c$, and
\[
L_t^{i+1}(b) = L_t(L_t^i(b)) = L_t(b') = L_t(t^i b - cf)
\]
\[
= t(t^i b - cf) \mod f = (t^{i+1}b - tcf) \mod f = (t^{i+1}b) \mod f=L_{t^{i+1}}(b).
\]
Hence the result follows by induction.

(4) Let $W$ be a $L_t$-invariant subspace of $V$ such that $0 < r := {\mathrm{dim}}(W) < d$. Choose some non-zero $w \in W$. Then the set
\[
\{ w,L_t w,L_t ^2w, \ldots , L_t ^r w \} \subset W
\]
is linearly dependent. Hence there exist elenents $a_0,a_1 ,\ldots,a_d$ in $\Fqn$, not all zero, such that
\[
\sum_{i = 0}^{d-1} a_i (L_t^i w) = 0.
\]
Let $a = \sum_{i = 0}^{d-1} a_i t^i$. Then
\[
\left (\sum_{i = 0}^{d-1} a_i L_t^i\right ) w = a(L_t)w = 0.
\]
By part (3) of this theorem, $a(L_t) = L_a$, and so
\[
a \circ_f w = 0.
\]
But $a \circ_f w = 0$ implies $a = 0$ or $w = 0$, a contradiction. Hence $L_t$ is irreducible.
\end{proof}

\begin{corollary}
The spread set of endomorphisms of left multiplication of elements in $\SSS_f$ is
$\{ a(L_t) ~|~ a \in \SSS_f \}$.
\end{corollary}

In \cite{JHJO1989} Jha and Johnson defined a semifield as follows.
\begin{theorem}
Let $T$ be an irreducible element of ${\mathrm{\Gamma L}}(d,K)$. Fix a $K$-basis $\{e_0,e_1, \ldots ,e_{d-1}\}$ of $V$. Define a multiplication on $V$ by
\[
a \circ b = a(T)b = \sum_{i=0}^{d-1} a_i T^i(b)
\]
where $a = \sum_{i=0}^{d-1} a_i e_i$. Then $\SSS_T = (V,\circ)$ defines a semifield.
\end{theorem}

The following theorem is an immediate consequence of the definition of $\SSS_T$ and Theorem \ref{thm:correspondence}.

\begin{theorem}\label{thm:correspondence2}
If $f$ is irreducible in $R$, and $L_{t,f}$  denotes the semilinear transformation $v \mapsto tv \mod f$, then
$\SSS_f = \SSS_{L_{t,f}}$.
\end{theorem}

Kantor and Liebler noted that conjugate semilinear transformations define isotopic semifields:
\begin{lemma}
Suppose $T = \phi^{-1} U \phi$ for some $\phi \in {\mathrm{\Gamma L}}(d,K)$, and let $\rho \in Aut(K)$ be the accompanying automorphism of $\phi$. Let $\SSS_T = (V,\circ)$, $\SSS_U = (V,\star)$. Then
\[
\phi(a \circ b) = a^{\rho} \star \phi(b).
\]
\end{lemma}

We will show that each semifield defined by a semilinear transformation is isotopic to some semifield $\SSS_f$ with
$f$ irreducible in some skew polynomial ring.

\begin{theorem}\label{thm:correspondence3}
Let $T$ be any irreducible element of ${\mathrm{\Gamma L}}(d,K)$ with automorphism $\sigma$. Then $T$ is ${\mathrm{GL}}(d,K)$-conjugate to $L_{t,f}$ for some irreducible $f \in R = K[t;\sigma]$, and hence $\SSS_T$ is isotopic to $\SSS_f$.
\end{theorem}

\begin{proof}

Identify $V$ with the set of polynomials of degree $\leq d-1$ in $R$ and choose some non-zero element $v\in V$. Consider the basis
\[
\{ v,Tv,T^2 v , \ldots , T^{d-1}v \},
\]
and define a transformation $\phi \in GL(d,K)$ by
\[
\phi(t^i) := T^i v,
\]
for $i = 0,1,\ldots,d-1$. Then there exist $f_i \in K$ such that
\[
T^d v = \sum_{i=0}^{d-1} f_i T^i v.
\]
It is left to the reader to verify that
\[
T \phi = \phi L_{t,f},
\]
where
\[
f = t^d - \sum_{i=0}^{d-1} f_i t^i \in R.
\]
As shown in Theorem \ref{thm:correspondence}, part (4), $L_{t,f}$ is irreducible if and only if $f$ is irreducible in $R$.
\end{proof}

\begin{theorem}\label{thm:correspondence4}
Let $f$ and $g$ be two monic irreducibles of degree $d$ in $R$. Then
\begin{enumerate}
	\item
		$L_{t,f}$ and $L_{t,g}$ are $GL(d,K)$-conjugate if and only if $f$ and $g$ are similar;
	\item
		$L_{t,f}$ and $L_{t,g}$ are ${\mathrm{\Gamma L}}(d,K)$-conjugate if and only if $f$ and $g^{\rho}$ are similar for some $\rho \in Aut(K)$.
\end{enumerate}
\end{theorem}

\begin{proof}
Suppose $L_{t,f} \phi = \phi L_{t,g}$ for some $\phi \in {\mathrm{\Gamma L}}(d,K)$, where $\phi$ has automorphism $\rho$. Let $\phi(1) = u$. Then
\[
\phi(t^i) = \phi (L_{t,g}^i (1)) = L_{t,f}^i \phi (1)
= L_{t,f}^i (u) = t^i u \mod f
\]
for all $i = 0,1,\ldots,d-1$. Now, with $g=t^d-\sum_{i=0}^{d-1} g_i t^i$, we have
\[
\phi L_{t,g} (t^{d-1}) = \phi(t^d \mod g) 
= \phi\left (\sum_{i=0}^{d-1} g_i t^i\right ) = \sum_{i=0}^{d-1} g_i^\rho \phi(t^i)
\]
\[
=\sum_{i=0}^{d-1} g_i^\rho (t^i u \mod f) = \left (\sum_{i=0}^{d-1} g_i^\rho t^i\right ) u \mod f 
= (t^d - g^{\rho})u \mod f.
\]
But as $L_{t,f} \phi = \phi L_{t,g}$, this is equal to
\[
L_{t,f} \phi(t^{d-1}) = L_{t,f} (t^{d-1}u \mod f).
\]
Let $t^{d-1}u = af + b$ where $\deg(b)<d$. Then
\[
L_{t,f} (t^{d-1}u \mod f) = L_{t,f}(t^{d-1}u - af)
\] 
\[
= (t^d u - taf) \mod f = t^d u \mod f.
\]
Hence
\[
(t^d - g^{\rho})u \equiv t^d u \mod f
\]
and so
\[
g^{\rho}u = 0 \mod f
\]
i.e. $f$ and $g^{\rho}$ are similar. If $\phi \in GL(n,K)$, then $\rho$ is the identity automorphism, and so $f$ and $g$ are similar.
\end{proof}

This provides an alternate proof of the following result proved by Dempwolff (\cite{Dempwolffprep}, Theorem 2.10), (Compare to Asano-Nakayama \cite{Asano}, Satz 13).
\begin{corollary}
Let $T$ and $U$ be two irreducible elements of ${\mathrm{\Gamma L}}(d,K)$, $K=\Fqn$, where the accompanying automorphism $\sigma$ of both $T$ and $U$ is a generator of $Gal(K,\Fq)$. Then
\begin{enumerate}
	\item
		$T$ and $U$ are ${\mathrm{GL}}(d,K)$-conjugate if and only if $T^n$ and $U^n$ have the same minimal polynomial over $\Fq$;
	\item
		$T$ and $U$ are ${\mathrm{\Gamma L}}(d,K)$-conjugate if and only if the minimal polynomials of $T^n$ and $U^n$ over $\Fq$ are $Aut(\Fq)$ conjugate.
\end{enumerate}
\end{corollary}

\begin{proof}

(1) By Theorem \ref{thm:correspondence3}, we may assume $T$ is $GL(d,K)$-conjugate to $L_{t,f}$, $U$ is $GL(d,K)$-conjugate to $L_{t,g}$, for some $f,g\in R$ irreducibles of degree $d$. Let $mzlm(f)=\hat{f}(t^n)$ for $\hat{f} \in \Fq[y]$, and suppose $\hat{f}(t^n) = af$. As $\sigma$ has order $n$, $T^n$ and $U^n$ are $\Fqn$-linear. We claim that $\hat{f}$ is the minimal polynomial of $L_{t,f}^n$ over $\Fq$, and hence the minimal polynomial of $T^n$ over $\Fq$. For any $v$,
\[
\hat{f}(L_{t,f}^n)v = \hat{f}(t^n) v \mod f
\]
\[
= v \hat{f}(t^n) \mod f = vaf \mod f = 0.
\]
Hence $\hat{f}(L_{t,f}^n) = 0$. Suppose now $\hat{h}(L_{t,f}^n) = 0$ for some $\hat{h} \in \Fq[y]$. Then
\[
\hat{h}(L_{t,f}^n)(1) = \hat{h}(t^n) \mod f = 0.
\]
But then $f$ divides $\hat{h}(t^n)$, and $\hat{h}(t^n)$ is in the centre of $R$, so $\hat{f}$ divides $\hat{h}$. Therefore $\hat{f}$ is the minimal polynomial of $L_{t,f}^n$ (and hence $T^n$) over $\Fq$ as claimed.

Similarly, if $mzlm(g)=\hat{g}(t^n)$, then $\hat{g}$ is the minimal polynomial of $L_{t,g}^n$, and $U^n$, over $\Fq$.

By Theorem \ref{thm:equivalence1} and Theorem \ref{thm:correspondence4}, $L_{t,f}$ and $L_{t,g}$ are $GL(d,K)$-conjugate if and only if $mzlm(f) = mzlm(g)$. Hence $T$ and $U$ are $GL(d,K)$-conjugate if and only if $T^n$ and $U^n$ have the same minimal polynomial over $\Fq$.

(2) Similarly, $T$ and $U$ are ${\mathrm{\Gamma L}}(d,K)$ conjugate if and only if $\hat{f}=\hat{g}^{\rho}$ for some $\rho \in Aut(\Fq)$, i.e. if and only if the minimal polynomials $T^n$ and $U^n$ over $\Fq$ are $Aut(\Fq)$ conjugate.
\end{proof}

As we know that these minimal polynomials are irreducible and have degree $d$ in $\Fq[y]$, this result of Dempwolff implies an upper bound on the number of conjugacy classes, and hence the number of isotopy classes:
\[
A(q,n,d) \leq N(q,d) = \#I(q,d).
\]

\section{Isotopisms between different skew-polynomial rings}
In this section, we consider the isotopism problem for semifields constructed from different skew polynomial rings.
The most general skew polynomial ring has the form $K[t;\sigma,\delta]$, with multiplication
defined by
$$
ta=a^\sigma t + a^\delta,
$$
where $\sigma$ is an automorphism of $K$ and $\delta$ is a $\sigma$-derivation, i.e. an additive map on $K$ such that
\[
(ab)^{\delta} = a^{\sigma}b^{\delta} + a^{\delta} b
\]
for all $a,b \in K$. For example, for any $x \in K$ the map
\[
\delta_x : a \mapsto x(a-a^{\sigma})
\]
is a $\sigma$-derivation. It is easily verified that for a finite field, every $\sigma$-derivation is of this form. Petit \cite{Petit} showed that these rings can also be used to define semifields with the same nuclei as those obtained from $K[t;\sigma]$.
However, as the following theorem of Jacobson shows, $K[t;\sigma,\delta_x]$ is isomorphic to $K[t;\sigma]$ for all $x$, and hence the semifields obtained are isotopic.

\begin{theorem}[Jacobson \cite{JacobsonBook} Prop 1.2.20:]
Let $R = K[t;\sigma]$ and $R' = K[t;\sigma,\delta_x]$ be skew-polynomial rings. Denote the multiplication in $R$ and $R'$ by $\circ$ and $\circ'$ respectively. Define a map $\phi:R \rightarrow R'$ by
\[
a(t) \mapsto a(t-x),
\]
where the evaluation of $f(t-x)$ occurs in $R'$ (i.e. $\phi(t^2) = (t-x) \circ' (t-x)$). Then the map $\phi$ is linear and
\[
\phi(a \circ b) = \phi(a) \circ' \phi(b)
\]
for all $a,b \in R$
\end{theorem}

\begin{proof}
Clearly by the definition of $\phi$, $\phi(t^i \circ t^j) = \phi(t^i) \circ' \phi(t^j)$ for all $i,j$, and $\phi(\alpha \circ \beta t^i) = \phi(\alpha) \circ' \phi(\beta t^i)$ for all $\alpha,\beta \in K$ and all $i$. Hence it suffices to show that 
\[
\phi(t \circ \alpha) = \phi(t) \circ' \phi(\alpha)
\]
for all $\alpha \in K$. Now
\[
\phi(t \circ \alpha) = \phi(\alpha^{\sigma} t) = \phi(\alpha^{\sigma}) \circ' \phi(t)
= \alpha^{\sigma} \circ' (t-x) = \alpha^{\sigma}t - x \alpha^{\sigma}
\]
while
\[
\phi(t) \circ' \phi(\alpha) = (t-x) \circ' \alpha 
= \alpha^{\sigma} t + x(\alpha - \alpha^{\sigma}) - x \alpha = \alpha^{\sigma}t - x \alpha^{\sigma}
\]
and the result holds.
\end{proof}

Note that defining the multiplication using remainder on \emph{left} division by $f$ also defines a semifield. However, the following theorems show that the semifields obtained are anti-isomorphic.

\begin{theorem}
\label{thm:antiisom} 
Let $R = K[t;\sigma]$ and $R' = K[t;\sigma^{-1}]$ be skew-polynomial rings. Denote the multiplication in $R$ and $R'$ by $\circ$ and $\circ'$ respectively. Then $R$ and $R'$ are anti-isomorphic via the map $\psi:R \rightarrow R'$ defined by
\[
\psi\left (\sum a_i t^i\right ) = \sum a_i^{\sigma^{-i}} t^i,
\]
i.e.
\[
\psi(a \circ b) = \psi(b) \circ' \psi(a)
\]
\end{theorem}

\begin{proof} For any $a,b$,
\begin{align*}
\psi(a \circ b) &= \psi\left (\sum_{i,j} a_i b_j^{\sigma^i} t^{i+j}\right )\\
								&= \sum_{i,j} a_i^{\sigma^{-i-j}} (b_j^{\sigma^i})^{\sigma^{-i-j}} t^{i+j}\\
								&= \sum_{i,j} (b_j^{\sigma^{-j}}) (a_i^{\sigma^{-i}})^{\sigma^{-j}}  t^{i+j}\\
								&= \sum_{i,j} \psi(b)_j (\psi(a)_i)^{\sigma^{-j}}  t^{i+j} \\
								&= \psi(b) \circ' \psi(a),
\end{align*}
as claimed.
\end{proof}

\begin{corollary}\label{thm:leftdivision} Let $R$, $R'$ and $\psi$ be as above. Let $f$ be irreducible in $R$. Then
\begin{enumerate}
	\item
		$\psi(f)$ is irreducible in $R'$;
	\item
		If $\SSS_f = R \mod Rf$, and $_{\psi(f)}\SSS' = R' \mod \psi(f)R'$, then $\SSS_f$ and $_{\psi(f)}\SSS'$ are anti-isomorphic.
\end{enumerate}
\end{corollary}

\begin{proof} (1) Clear, for if $\psi(f) = \psi(a) \circ' \psi(b)$, then by the previous theorem, $f = b \circ a$. But then $a$ or $b$ must be a unit, and as $\psi$ preserves degrees, $\psi(a)$ or $\psi(b)$ must be a unit.

(2) We claim that $\psi$ is an anti-isomorphism from $\SSS_f$ to $_{\psi(f)}\SSS'$. Clearly $\psi$ is a bijective linear map. We need to show that
$$\psi (a\circ_f b) = \psi(b) _{\psi(f)}\circ' \psi(a),$$ 
where
$$a\circ_f b = a\circ b \mod_r ~f, ~\mbox{and}$$
$$\psi(b) _{\psi(f)}\circ' \psi(a) = \psi(b)\circ' \psi(a) \mod_l ~\psi(f).$$
Let $a \circ b = u \circ f + v$, where $\deg(v)<d = \deg(f)$. Then using the above theorem we obtain
\begin{align*}
\psi (a\circ_f b)	&= \psi(v) \\
									&= \psi(a \circ b - u \circ f) = \psi(a \circ b) - \psi(u \circ f)\\
									&= \psi(b) \circ' \psi(a) - \psi(f) \circ' \psi(u) = \psi(b) \circ' \psi(a) \mod_l ~\psi(f)\\
									&=\psi(b) _{\psi(f)}\circ' \psi(a),
\end{align*}
as claimed.
\end{proof}

\begin{remark}
It is not clear when different skew polynomial rings $R = K[t;\sigma]$ and $R' = K[t;\sigma']$ define isotopic semifields. It is a necessary condition that $\sigma$ and $\sigma'$ have the same order. The following observation of Kantor and Liebler (in \cite{KALI2008}) gives a result in this direction.

{\it If $T$ is an irreducible semilinear transformation with automorphism $\sigma$, then $T^{-1}$ is an irreducible semilinear transformation with automorphism $\sigma^{-1}$, and $\SSS_T$ and $\SSS_{T^{-1}}$ are isotopic.} (See\cite{KALI2008}, Remark 4.1, where the statement is made in terms of projective planes.)

This implies that every semifield $\SSS_f$ for $f \in K[t;\sigma]$ is isotopic to $\SSS_{\bar{f}}$ for some $\bar{f} \in K[t;\sigma^{-1}]$. In fact it can be shown that $\bar{f}$ is the reciprocal of $f$.
Hence the total number of isotopy classes defined by degree $d$ irreducibles in $\Fqn[t;\sigma]$ for all $\sigma$ fixing precisely $\Fq$ is upper bounded by $\frac{\phi(n)}{2} M(q,d)$, when $n\neq 2$, 
where $\phi$ is Euler's totient function.
\end{remark}

\section{New and existing bounds for $A(q,n,d)$}

Let $N(q,d) = \# I(q,d)$, where $I(q,d)$ is the set of monic irreducibles of degree $d$ in $\Fq[y]$. This number is well known and equal to $\frac{1}{d}\sum_{s|d} \mu(s)q^{d/s}$, where $\mu$ denotes the Moebius function. Following the notation of \cite{JOMAPOTR2009}, we can write this as $N(q,d) = \frac{q^d - \theta}{d}$, where $\theta$ denotes the number of elements of $\Fqd$ contained in a proper subfield $\FF_{q^e}$.

Let $A(q,n,d)$ denote the number of isotopy classes of semifields of order $q^{nd}$ defined by the skew polynomial ring $\Fqn[t;\sigma]$, (or equivalently, semilinear transformations with automorphism $\sigma$), with
\[
(\#Z,\#N_l,\#N_m,\#N_r)  = (q,q^n,q^n,q^d).
\]

In \cite{JOMAPOTR2009}, the authors consider cyclic semifields two-dimensional over their left nucleus, with right and middle nuclei isomorphic to $\FF_{q^2}$t. The above defines the opposite semifield to those in this paper. Hence they are considering semifields $\SSS_f$, where $f \in \FF_{q^2}[t;\sigma]$ is an irreducible of degree $d$ (denoted by $n$ in their paper). They prove the lower bound
\[
A(q,2,d) \geq \frac{q^d - \theta}{2dhq(q-1)}
\]
where $q = p^h$.

In \cite{KALI2008}, the authors obtain an upper bound 
\[
A(q,n,d) \leq q^d - 1.
\]
They also obtained an upper bound for the total number of isotopy classes of semifields of order $q^{nd}$ obtained from semilinear transformations of order $q^{nd}$:
\[
nd q^{nd/2} log_2(q).
\]
The bounds for $A(q,n,d)$ that are proved in this paper arise from the following isotopism criteria.
In Theorem \ref{thm:similar}, we proved that if $f$ and $g$ are irreducibles of degree $d$ in $\Fqn[t;\sigma]$, with $gu\equiv 0 \mod f$ for some $u\in \Fqn[t;\sigma]$, then $\SSS_f$ and $\SSS_g$ are isotopic, and
\[
(a \circ_g b)^{H}= a \circ_f  b^{H}
\]
where $b^H = b\circ_f u$.
Next we have shown that this condition is equivalent to $mzlm(f)=mzlm(g)$ (Theorem \ref{thm:equivalence1}). This leads to the following upper bound
(this also follows from the result of Dempwolff \cite{Dempwolffprep}, by Theorem \ref{thm:correspondence3} above):
\[
A(q,n,d) \leq N(q,d) = \frac{q^d - \theta}{d}.
\]
We improve this bound by showing that $\SSS_f$ and $\SSS_g$ are also isotopic when 
$$ \lambda^{-d} \hat{f}^{\rho}(\lambda y)=\hat{g},$$ 
for some $\lambda \in \Fq^\times$, $\rho \in Aut(\Fq)$, where $\hat{f}=mzlm(f)$ and $\hat{g}=mzlm(g)$ (Theorem \ref{thm:bound}). This leads to the upper bound:
\[
A(q,n,d) \leq M(q,d),
\]
where $M(q,d)$ denotes the number of orbits in $I(q,d)$ under the action of $G$ defined in the introduction.

Note that if $q = p^h$ for $p$ prime, then
\[
\frac{q^d - \theta}{dh(q-1)} \leq M(q,d) \leq \frac{q^d - \theta}{d}.
\]

\textbf{Example:} For $q = \{2,3,4,5\}$, $n = d = 2$, the upper bounds $M(q,d) = \{1,2,1,3\}$ are tight by computer calculation.

\textbf{Example:} If $q$ is prime, and $(q-1,d) = 1$, then $M(q,d) = \frac{N(q,d)}{q-1}$.

\begin{remark}
To produce a specific example of every isotopy class of cyclic semifields, it suffices to find representatives $\hat{f_i}$ of each $G$-orbit of $I(q,d)$. We form the skew-polynomials $\hat{f_i}(t^n)$, and calculate a particular irreducible divisor $f_i$ of each, using for example the algorithm of Giesbrecht \cite{Giesbrecht1998}. Then the semifields $\SSS_{f_i}$ are representatives of each isotopy class.
\end{remark}

{\bf Acknowledgement}\\
The authors thank W. M. Kantor and the anonymous referee for their many helpful suggestions in improving this paper, and for bringing the papers of Petit and Wene to their attention.

\bibliographystyle{plain}

\end{document}